\documentclass[11pt]{amsart}
\usepackage[top=3.5cm,bottom=3.5cm,left=4cm,right=4cm]{geometry}
\usepackage{amsmath,amscd,amssymb,amsthm}
\usepackage[english]{babel}
\usepackage{lipsum}
\usepackage{mathtools}
\usepackage{enumerate}
\usepackage{setspace}
\usepackage{mathrsfs}
\usepackage[numbers]{natbib}
\usepackage[hidelinks]{hyperref}

\DeclareMathOperator*{\GF}{{GF}}

\def\vB{\mathbb{B}}

\def\vF{\mathbb{F}}
\def\vS{\mathbb{S}}
\def\vN{\mathbb{N}}
\def\vD{\mathbb{D}}
\def\vZ{\mathbb{Z}}

\def\vR{\mathbb{R}}
\def\vQ{\mathbb{Q}}

\def\OO{\mathcal{O}}

\def\cO{\mathcal{O}}

\def\vE{\mathbb{E}}

\def\p{\mathfrak{p}}

\newtheorem{teo}{Theorem}
\newtheorem{lemma}[teo]{Lemma}
\newtheorem{cor}[teo]{Corollary}
\newtheorem{prop}[teo]{Proposition}
\theoremstyle{definition}

\newtheorem{example}[teo]{Example}
\newtheorem{rem}[teo]{Remark}

\newtheorem*{claim*}{Claim}

\makeindex

\author[A. Ferraguti]{Andrea Ferraguti}
\address{Institute of Mathematics\\
University of Zurich\\
Winterthurerstrasse 190\\
8057 Zurich, Switzerland\\
}
\email{andrea.ferraguti@math.uzh.ch}

\author[G. Micheli]{Giacomo Micheli }
\address{Institute of Mathematics\\
University of Zurich\\
Winterthurerstrasse 190\\
8057 Zurich, Switzerland\\
}
\email{giacomo.micheli@math.uzh.ch}
\thanks{The second author was supported in part by  Swiss National Science Foundation grant number 149716 and \emph{Armasuisse}.}

\date{}
\title{On Mertens-Ces\`aro Theorem for number fields}
\begin{document}

\begin{abstract}
Let $K$ be a number field with ring of integers $\mathcal O$. After introducing a suitable notion of density for subsets of $\mathcal O$, generalizing that of natural density for subsets of $\vZ$, we show that  the density of the set of coprime $m$-tuples of algebraic integers is ${1/\zeta_K(m)}$, where $\zeta_K$ is the Dedekind zeta function of $K$. This generalizes a result found independently by Mertens (1874) and Ces\`aro (1883) concerning the density of coprime pairs in $\vZ$.
\end{abstract}
\maketitle

\smallskip
\noindent \textbf{Keywords:} Number fields; Algebraic integers; Natural density; Mertens-Ces\`aro Theorem; Zeta function.

\smallskip
\noindent \textbf{2010 Mathematical Subject Classification :} 11R04, 11R45.

\section{Introduction}
In 1874 Mertens  proved that the natural density of the set of coprime pairs of rational integers is $1/\zeta(2)$, where $\zeta$ is the Riemann zeta function \cite{mertens}.
In 1881 Ces\`aro independently asked the same question in \cite{CESAROQ} and provided the solution two years later in \cite{CESARO}, getting the same result as Mertens.
Another proof of this result is presented in the book by Hardy and Wright \cite[Theorem 330]{HW}, while a generalization to the case of $m$-tuples of integers has been more recently given in \cite{Nymann}.

If one tries to extend the formulation of the theorem to the case of algebraic integers, one encounters some obstructions from the very beginning. In the next paragraphs the reader can find some of the motivations that led to our approach to the problem, especially concerning the definition of the density for a subset of the ring of algebraic integers $\OO$ of a number field $K$.

Indeed, for the case of $\vZ$, there exists  a ``canonical'' way to compute the density of a set
 $A\subseteq\vZ$:
 this can be in fact defined as the limit in $B$ (if it
  exists) of the sequence $|A\cap [-B,B[|/(2B)$. This
   definition extends to the density of a set $A\subseteq \vZ^m$ by considering the limit of the sequence $|A\cap [-B,B[^m|/(2B)^m$ as $B$ goes to infinity.
This definition characterizes the probability that, given the $m$-dimensional hypercube of large side $B$ centred in the origin, a uniformly random selected integer point has all relatively prime entries.

What can actually be done in the setting of algebraic integers is to consider the analogous problem for the set of
$m$-tuples of ideals of $\OO$ using a suitable definition of density involving the norm function. Very interesting results in this direction can be found in \cite{BS}. On the other hand, if we want a proper
 generalization of Mertens-Ces\`aro Theorem to $\OO$ (and not to the set of
  ideals of $\OO$) the approach presented in \cite{BS} does not apply:
   indeed, given a large bound $B$, there might be infinitely many
    elements of norm at most $B$ (contrary to what happens in
     the case of $\vZ$). Therefore, not only this definition of density
      for sets of ideals of $\OO$ cannot extend to a definition of
       density for $\OO$, but also the analogous probability interpretation that one has over $\vZ$ is missing.
%
%

A \emph{non canonical} definition for the density of a subset $A\subseteq\OO$ is obtained by considering a $\vZ$-isomorphism $\alpha:\OO\rightarrow \vZ^n$ ($n$ being  the degree of the extension $K\supseteq \vQ$) and then by computing the density of $\alpha(A)\subseteq\vZ^n$ as previously described. The resulting density is then dependent on the choice of $\alpha$ (that is equivalent to a choice of a $\vZ$-basis for $\OO$), but  extends to $A\subseteq \OO^m$ componentwise, as one would expect by considering the limit of the sequence $|\alpha(A)\cap[-B,B[^{mn}|/(2B)^{mn}$.
Using this definition of density for the set $E\subseteq \OO^m$ of coprime $m$-tuples and a similar strategy to the one presented in \cite{Maze20} for the case of unimodular matrices over $\vZ$, the following turns out to be true:
\begin{itemize}
\item the density $d$ of $E$ can be computed;
\item $d$ is \emph{independent} on the choice of the embedding $\alpha$ (i.e. independent of the choice of the $\vZ$-basis for $\OO$);
\item $d$ equals $1/\zeta_K(m)$, where $\zeta_K(m)$ is the Dedekind zeta function of the number field $K$.
\end{itemize}
This completely generalizes Mertens-Ces\`aro Theorem to the case of number fields. It is very interesting to note that this result matches the one presented in \cite[Theorem 4.1]{BS}, that was obtained in the context of ideals of $\OO$.
\subsection*{Outline of the proof}
Let us now briefly describe the strategy we use to compute the above mentioned density in the general case of a subset $E\subseteq \vZ^M$ (In our case $M=nm$).
First, we find a family $\{E_t\}_{t\in\vN}$ of subsets of $\vZ^M$  with the following properties:
\begin{itemize}
\item we are able to compute the density of $E_t$ for every $t$ (Lemma \ref{fondlemma});
\item $E_{t+1}\subseteq E_t$;
\item $\bigcap_{t\in \vN} E_t=E$.
\end{itemize}
Then we verify that the family of sets $\{E_t\}_{t\in \vN}$ approximates the set $E$ \emph{in density} in the sense that the sequence of densities of $E_t\setminus E$ converges to zero as $t$ tends to infinity.
Under these assumptions we are able to prove that $\lim_{t\rightarrow \infty} \vD(E_t)=\vD(E)$ (Theorem \ref{MAIN}).

\subsection{Notation}
We say that the ideals $I_1,\dots,I_l$ are coprime if $\sum_j I_j=R$; we say that the elements $a_1,\ldots, a_s\in R$ are coprime if the ideals $(a_1),\ldots,(a_s)$ are coprime.
Let $K$ be a number field of degree $n$ and $\OO$ its ring of algebraic integers. Let  $\vE=\{\mathbf{e}_i\}_{i=1}^n$ be a $\vZ$-basis for $\OO$.
Define
\[\OO[B,\vE]=\left\{\sum^{n}_{i=1} a_i \mathbf{e}_i\;|\; a_i\in [-B,B[\cap \vZ\right\}.\]
Later on in the paper we will just write $\OO[B]$ since the basis will be understood.
For $p$ a prime number, we denote by $S_p=\{\p_1^{(p)},\dots, \p_{\lambda_p}^{(p)}\}$ the set of distinct prime ideals lying over $p$ (in particular we have that $\prod^{\lambda_p}_{j=1}\p_{j}^{(p)}$ is the radical of the ideal generated by $p$). Let $d_j^{(p)}$ be the inertia degree of $\p_{j}^{(p)}$ (i.e. $\dim_{\vF_p}(\OO/\p_{j}^{(p)})$) and denote by $D_p$ the integer $\sum^{\lambda_p}_{j=1} d_j^{(p)}$. Let $d$ be a positive integer, let us denote by $\GF(p,d)$ the finite field of order $p^d$. Define
\[R_p:=\prod^{\lambda_p}_{j=1} \OO/\p_j^{(p)}\cong \prod^{\lambda_p}_{j=1} \GF(p,d_j^{(p)}).\]
For $z=(z_1,\dots,z_m)$ an element of $\OO^m$, we denote by $I_z$ the ideal generated by the set $\{z_1,\dots,z_m\}$.
If $\vF$ is a field we denote by $\vF^*$ its multiplicative group.
\section{A definition of the density for  $\OO^m$}\label{sec:density}

Let $\vE$ be a $\vZ$-basis for $\OO$. Our goal is to define a notion of density (which will in general depend on the choice of $\vE$) for a subset $T$ of $\OO^m$.
We define the \emph{upper density of $T$ with respect to $\vE$} to be
\[\overline\vD_\vE(T)=\limsup_{B\rightarrow\infty} \frac{|\OO[B,\vE]^m\cap T|}{(2B)^{mn}}\]
and the \emph{lower density of $T$ with respect to $\vE$} as
\[\underline\vD_\vE(T)=\liminf_{B\rightarrow\infty} \frac{|\OO[B,\vE]^m\cap T|}{(2B)^{mn}}.\]
We say that $T$ has \emph{density $d$ with respect to $\vE$} if
\[\overline\vD_\vE(T)=\underline\vD_\vE(T)=:\vD_\vE(T)=d.\]
Whenever this density is independent of the chosen basis $\vE$, it is consistent to denote the density of a set $T$ by $\vD(T)$ without any subscript.
\begin{rem}
First observe that $d\in[0,1]\subseteq \vR$ by construction.
The main idea behind this definition of density is the same that one has over $\vZ$: the only difference is that the way in which we cover the entire set (in this case $\OO$) is not canonical but depends on the basis $\vE$.
\end{rem}
\begin{example}
Let us show with an example that choosing different bases for $\OO$ could yield different densities for the same subset $T\subseteq \OO$. Let $K=\vQ(i)$, so that $\OO=\vZ[i]$. Let $T=\{x+iy\in \OO\colon x,y>0\}$. If $\vE=\{1,i\}$, clearly $|\OO[B,\vE]\cap T|=(B-1)^2$, which gives $\vD_\vE(T)=1/4$. On the other hand, choosing as a basis $\vE'=\{1,-1+i\}=\{\mathbf e_1,\mathbf e_2\}$ we have that $T=\{x\mathbf e_1+y\mathbf e_2\in \OO\colon x,y>0,\,\, x>y\}$. Therefore $|\OO[B,\vE']\cap T|=(B-1)(B-2)/2$, which shows that $\vD_{\vE'}(T)=1/8$.
\end{example}
Let $E\subseteq\OO^m$ be the set of coprime $m$-tuples, i.e. the elements $z\in \OO^m$ for which $I_z=\OO$.
A corollary of our final result (Theorem \ref{MAIN}) is that the density of $E$ is actually independent of the basis $\vE$: even if the choice of the covering of $\OO^m$ is not canonical (it depends in fact on the chosen $\vZ$-basis for $\OO$) the  density of $E$ \emph{is}.

\section{Proof of the main result}\label{sec:proof}
Let $\vS$ be a finite set of prime numbers.
Let $E_\vS$ be the set of $m$-tuples $z=(z_1,\dots,z_m)$ in $\OO^m$ such that
the ideal $I_z$ is coprime with every $p\in \vS$.
\begin{rem}
Equivalently, one checks that
\[E_\vS=\{z\in \OO^m\;|\; I_z +\p_j^{(p)}=\OO\quad \forall p\in \vS\quad\text{and}\quad \forall j\in\{1,\dots,\lambda_p\}\}\]
by observing that $(p)\subseteq \prod_j \p_j^{(p)}$ and the $\p_j^{(p)}$ are maximal.
\end{rem}
Let $\psi_p: (\OO/(p))^{m}\rightarrow R_p^{m}=(\prod^{\lambda_p}_{j=1} \OO/\p_j^{(p)})^{m}$ be the morphism induced by the projection
$\OO/(p)\twoheadrightarrow \prod^{\lambda_p}_{j=1} \OO/\p_j^{(p)}$.
Recall that $D_p=\sum^{\lambda_p}_{j=1} d^{(p)}_{j}$.
In the following lemma and in Proposition \ref{teocard} we will consider the surjection
\[\pi: \OO^{m}\longrightarrow \left(\prod_{p\in \vS}R_p\right)^{m}\eqqcolon T\]
induced by the quotient maps $\OO\rightarrow \OO/\p_j^{(p)}$.
It is easy to prove the following
\begin{lemma}\label{usefullemma}
We have
\[{E_\vS}=\pi^ {-1}\left(\prod_{p\in \vS}\prod^{\lambda_p}_{j=1} \left(\left(\OO/\p_j^{(p)}\right)^{m}\setminus \{0\}\right)\right).\]
\end{lemma}
\begin{prop}\label{teocard}
Let $q$ be a positive integer, $\vE$ a $\vZ$-basis for $\OO$, $\vS$ a finite set of prime numbers and $N=\prod_{p\in \vS}p$. Then
\[|E_\vS\cap \OO[qN]^{m}|=(2q)^{mn}\prod_{p\in \vS}\left(p^{nm-  mD_p} \prod^{\lambda_p}_{j=1}(p^{d_j^{(p)}m} -1)\right)\]
where $\OO[qN]^{m}$ is the set of $m$-tuples of elements of $\OO[qN]$.
\end{prop}

\begin{proof}
The key point is to decompose the map $\pi$.
For the rest of the proof, the reader may refer to the following diagram:
\[
\begin{CD}
\OO^{m} @>\pi_N >>(\OO/(N))^{m}@>\overline\psi >> T \\
@.                  @|              @|\\
@.                  (\prod_{p\in \vS}\OO/(p) )^{m} @>{\psi}>> (\prod_{p\in \vS} R_p)^{m}
\end{CD}
\]
where $\pi_N$ is the quotient map, $\psi=(\dots,\psi_p,\dots )$ and $\overline \psi$ is its obvious extension to
$(\OO/(N))^{m}$ obtained by applying the Chinese Remainder Theorem to primes in $\vS$. Notice then that $\pi= \overline{\psi}\circ \pi_N$.
Our strategy to prove the result is to compute the cardinality of the fibers of $\psi$ and the intersection
of the fibers of $\pi_N$ with $\OO[qN]$:
\begin{itemize}
\item Observe that $\psi_p: (\OO/(p))^{m}\rightarrow R_p^{m}$ is a surjective morphism of $\vF_p$-vector spaces, therefore $|\psi_p^{-1}(y_p)|=|\ker(\psi_p)|=p^{nm-  mD_p}$ for all $y_p\in R_p^{m}$. It follows that $|\overline \psi^{-1}(y)|=\prod_{p\in \vS}|\psi_p^{-1}(y_p)|=\prod_{p\in \vS}p^{nm-  mD_p}$ for all $y\in (\OO/(N))^{m}$.
\item Let $\overline z=(\overline z_j)_j\in (\OO/(N))^{m}$ and $z=(z_j)_j\in \OO^{m}$.
Write \[\overline z_j=\left(\sum^{n}_{t=0}r_t^{j}\pi(\mathbf{e}_t)\right)\]
for some unique $0\leq r_t^{j}<N$ in $\vZ$. Observe that existence and uniqueness of the $r_t^j$ follow from the fact that $\OO/(N)$ is a free $\vZ/N\vZ$-module of rank $n$ with basis $\{\pi(\mathbf{e}_t)\}$.
It follows that $\pi_N(z)=\overline z$ if and only if
\[z_j=\sum^{n}_{t=0}(r_t^{j}+ l^{j}_t N)\mathbf{e}_t\]
for some $l^{j}_t\in \vZ$.
We conclude then that \[|\OO[qN]^{m}\cap \pi_N^{-1}(z)|= (2q)^{mn}\]  since the $r_t^{j}$ are fixed by the condition $\pi_N(z)=\overline z$ and $l^{j}_t\in [-q,q[\cap \vZ$ for each index $j,t$.
\end{itemize}

Let us now complete the proof.
By Lemma \ref{usefullemma} we have that
\begin{equation}\label{eqlemm}
{E_\vS}\cap \OO[qN]^m=\pi^{-1}\left(\prod_{p\in \vS}\prod^{\lambda_p}_{j=1} \left(\left(\OO/\p_j^{(p)}\right)^m\setminus \{0\}\right)\right)\cap \OO[qN]^{m}
\end{equation}

In order to simplify the notation, define
\[H:=\psi^{-1}\left(\prod_{p\in \vS}\prod^{\lambda_p}_{j=1} \left(\left(\OO/\p_j^{(p)}\right)^m\setminus \{0\}\right)\right)\]
so that $E_\vS=\pi_N^{-1}(H)$ by Lemma \ref{usefullemma}.
Since $\pi=\psi\circ \pi_N$, Equation (\ref{eqlemm}) reads
\[{E_\vS}\cap \OO[qN]^m=\pi_N^{-1}\left(H\right)\cap \OO[qN]^{m}.\]
Therefore
\[|\pi_N^{-1}(H)\cap \OO[qN]^{m}|=(2q)^{mn} |H|\] and
\[|H|=\prod_{p\in \vS}\left(p^{nm- D_p m} \prod^{\lambda_p}_{j=1}\left|\left(\OO/\p_j^{(p)}\right)^m\setminus \{0\}\right|\right).\]
Thus,
\[|{E_\vS}\cap \OO[B]|=(2q)^{mn}\prod_{p\in \vS}\left(p^{nm-  mD_p} \prod^{\lambda_p}_{j=1} (p^{d_j^{(p)}m} -1)\right).\]

\end{proof}

Before we proceed, let us recall the following elementary calculus fact

\begin{lemma}\label{calclemma}
Let $\{a_B\}_{B\in\vN}$ be a sequence of real numbers and $N$ a positive integer. Then
\[\lim_{B\rightarrow\infty}a_B=c\:\Leftrightarrow \: \lim_{q\rightarrow\infty}a_{r+qN}=c \quad \forall r\in\{0,\dots,N-1\}.\]
\end{lemma}

\begin{lemma}\label{fondlemma}
In the notation previously described we have
\[\vD(E_\vS)=\vD_\vE(E_\vS)= \prod_{p\in \vS}\prod^{\lambda_p}_{j=1} \left(1-\frac{1}{p^{d_j^{(p)}m}}\right).\]
\end{lemma}
\begin{proof}
Let \[a_B:=\frac{|\OO[B]^{m}\cap E_\vS|}{(2B)^{mn}}.\]
Recall that $N=\prod_{p\in \vS}p$. Let \[D:=\prod_{p\in \vS}\prod^{\lambda_p}_{j=1} \left(1-\frac{1}{p^{d_j^{(p)}m}}\right).\] We first show that $a_{qN}=D$. By Proposition \ref{teocard} we have that
\[a_{qN}=\frac{|\OO[qN]^{m}\cap E_\vS|}{(2qN)^{mn}}=\]
\[=\frac{(2q)^{mn}\prod_{p\in \vS}p^{nm-  mD_p} \prod^{\lambda_p}_{j=1} (p^{d_j^{(p)}m} -1)}{(2qN)^{mn}}.\]
By cancelling common factors in numerator and denominator and writing $D_p$ according to its definition we get
\[a_{qN}=\prod_{p\in \vS}p^{-m\sum^{\lambda_p}_{j=1} d_j^{(p)}} \prod^{\lambda_p}_{j=1} (p^{d_j^{(p)}m} -1)\]
and bringing $p^{-m\sum^{\lambda_p}_{j=1} d_j^{(p)}}$ inside the products it follows that
\[a_{qN}=\prod_{p\in \vS} \prod^{\lambda_p}_{j=1}\left(1 -\frac{1}{p^{d_j^{(p)}m}}\right).\]
We are now ready to prove that
\[\lim_{B\rightarrow\infty} a_{B}=D.\]
Thanks to lemma \ref{calclemma} it will be enough to show that
\[\lim_{q\rightarrow\infty}a_{r+qN}=D\]
for all $r\in \{0,\dots,N-1\}$.
Indeed
\[a_{qN} \cdot \left(\frac{(2qN)}{2r+2qN}\right)^{mn}<a_{r+qN}<a_{(q+1)N} \cdot \left(\frac{(2(q+1)N)}{2r+2qN}\right)^{mn}.\]
By passing to the limit in $q$ the claim follows.
\end{proof}
\begin{rem}
It is immediate to observe that the density of $E_\vS$ is independent of the chosen basis $\vE$.
\end{rem}

We are now in a position to formulate and prove the main result.
\begin{teo}\label{MAIN}
Let $m$ be a positive integer and $K$ be a number field. Let $\OO$ be the ring of integers of $K$.
The density of the set $E$ of coprime $m$-tuples of $\OO$ is
\[\vD(E)=\frac{1}{\zeta_K(m)}\]
where $\zeta_K$ is the Dedekind zeta function of the number field $K$.
\end{teo}
\begin{rem}\label{emme1}
Let $p_1,\dots,p_t$ be the first $t$ rational primes. We define $\vS_t=\{p_1,\dots,p_t\}$ . 
The reader should observe that one has the inclusion $E\subseteq E_{\vS_t}$ and therefore
\[0\leq \vD_\vE(E)\leq \overline\vD_{\vE}(E)\leq \vD(E_{\vS_t}).\]
As a consequence one has that in the case $m=1$ Theorem \ref{MAIN} follows by passing to the limit $t\rightarrow \infty$ in the above inequality and recalling that the Dedekind zeta function of $K$ has a pole at $1$.  As expected in fact, the group of units of the ring of integers has density zero in any basis. Observe that this is the special case $k=1$ of \cite[Corollary 4.2]{FrancCell}. A more extensive description of additive representations of elements in the unit group can be found in \cite{FabriU}.
\end{rem}
\begin{rem}
Notice that the argument of \ref{emme1} does not lead to the conclusion in the case $m>1$, since it provides just an upper bound (uniform in $\vE$) for $\overline\vD_\vE(E)$.
\end{rem}
Before starting the proof let us recall the following theorem, which we will use as a fundamental tool.
\begin{teo}\label{latticepoints}
 Let $S\subseteq \vR^M$ be a bounded set whose boundary $\partial S$ can be covered by the images of at most $W$ maps $\phi\colon [0,1]^{M-1}\to \vR^M$ satisfying Lipschitz conditions
	$$|\phi(x)-\phi(y)|\leq L|x-y|$$
	for the Euclidean norm. Then $S$ is measurable. Let $V=\text{vol}(S)$.\\
	Let $\Lambda\subseteq \vR^M$ be a full-rank lattice and
	$$\lambda_1\coloneqq \inf\{|v|\colon v\in \Lambda\setminus\{0\}\}$$
	be its first successive minimum. Then
	$$\left||\Lambda\cap S|-\frac{V}{\det\Lambda}\right|\leq c W\left(\frac{L}{\lambda_1}+1\right)^{M-1}$$
	for a constant $c$ depending only on $M$.
\end{teo}
\begin{proof}
 See \cite[Lemma 2]{masval}.
\end{proof}
Next we are going to deduce from Theorem \ref{latticepoints} the particular case that we will use in the proof of Theorem \ref{MAIN}.\\
\begin{prop}\label{corollario}
Let $K$ be a number field of degree $n$ with ring of integers $\OO$. Let $I$ be an ideal of $\OO$. Then
\[\left||(I\cap \OO[B])^m|-\frac{(2B)^{nm}}{N(I)^m}\right|\leq c\left(\frac{2B}{c_1N(I)^{1/n}}+1\right)^{mn-1}\]
for every $B\in \vN$, where $N(I)$ denotes the norm of $I$ and the constants $c,c_1$ are independent of $B$ and of $I$.
\end{prop}
\begin{proof}
Recall that there is a canonical embedding of $\OO$ into $\vR^n$: if $\sigma_1,\ldots,\sigma_r$ are the real embeddings $K\to \vR$ and $\sigma_{r+1},\ldots,\sigma_{r+2s}=\sigma_n$ are the complex ones labeled such that $\sigma_{r+i}=\overline{\sigma_{r+s+i}}$, then the map $\tau\colon\OO\to \vR^n$ defined by $x\mapsto (\sigma_1(x),\ldots,\sigma_r(x),\sigma_{r+1}(x),\ldots,\sigma_{r+s}(x))$ embeds $\OO$ as a full-rank lattice in $\vR^n$, where each $\sigma_{r+i}$ is thought as an embedding into $\vR^2$. The map $\tau$ induces an embedding $\tau^m\colon\OO^m\to \vR^{mn}$. The image of $\OO^m$ inside $\vR^{mn}$ via $\tau^m$ is again a full-rank lattice.
Let $\alpha_\vE:\OO\longrightarrow \vZ^n$ be the isomorphism of $\vZ$-modules
given by $\alpha_\vE\left(\sum^n_{i=1} x_i \mathbf{e}_i\right)=(x_1,\dots,x_n)$. Let $\alpha_{\vE}^m\colon \OO^m\to \vZ^{mn}$ be the isomorphism induced by $\alpha_{\vE}$.
Consider the following commutative diagram
\[
\begin{CD}
\OO^m @>\tau^m>> \vR^{mn}\\
@VV \alpha_\vE^m V @AA A A\\
\vZ^{mn} @>\iota>> \vR^{mn}
\end{CD}
\]
where $\iota$ is the inclusion map and $A$ is the unique $\vR$-linear map which makes the diagram commute.
The idea now is to apply Theorem \ref{latticepoints} with $\Lambda=(\iota\circ\alpha_\vE^m)(I^m)\subseteq \vR^{mn}$ and $S$ the cube of side $2B$ centered in the origin, so that $W=2mn$ and $L=2B$ in the notation of the theorem. Here by $I^m$ we mean the cartesian product of $m$ copies of $I$ inside $\OO^m$.\\
We first need a lower bound for the first successive minimum of $(\iota\circ\alpha_\vE^m)(I^m)$. To do this we can clearly assume $m=1$ since the first successive minimum of a lattice $\Lambda\subseteq \vR^n$ coincides with that of $\Lambda^m\subseteq  \vR^{mn}$.
Let $v$ be a vector realizing the first successive minimum of $(\iota\circ\alpha_\vE)(I)$ with respect to the euclidean norm $|\cdot |$. By Lemma 5 of \cite{masval}, the first successive minimum of $\tau(I)$ is greater or equal than $N(I)^{1/n}$.
Since $A(v)\in \tau(I)$ we have that
\[N(I)^{1/n}\leq |A (v)|\leq \|A\| | v| \]
where $\|A\|$ is defined by $\sup_{|w|=1}|A(w)|$. This shows that the first successive minimum of $(\iota\circ\alpha_{\vE})(I)$ is greater or equal than $c_1 N(I)^{1/n}$, where $c_1\coloneqq 1/\|A\|$ is independent of $B$ and of $I$.\\
Now the claim follows by applying Theorem \ref{latticepoints} together with the fact that
$$\det(\alpha_\vE^m(I^m))=\det(\alpha_{\vE}(I))^m=[\vZ^n:\alpha_\vE(I)]^m=[\OO\colon I]^m=N(I)^m$$
and observing that $|(\iota\circ\alpha_\vE^m)(I^m)\cap [-B,B[^{mn}|=|I^m\cap \OO[B]^m|=|(I\cap\OO[B])^m|$.
\end{proof}
\begin{proof}[Proof of Theorem \ref{MAIN}]
We already proved the Theorem in the case $m=1$ in Remark \ref{emme1}, therefore let us suppose $m>1$.
Let $t$ be a positive integer, $\vS_t$ the set consisting of the first $t$ prime numbers and define $E_t=E_{\vS_t}$. Observe that, since $E_t\supseteq E$ we have
\[\overline\vD_\vE(E)\leq \overline\vD(E_t)=\vD(E_t).\]
By letting $t$ run to infinity we get
\[\overline\vD_\vE(E)\leq \frac{1}{\zeta_K(m)}.\]
In order to show the opposite inequality observe that
\begin{equation}\label{equazionedensita}
\vD(E_t)-\overline\vD_\vE(E_t\setminus E)\leq\underline\vD_\vE(E).
\end{equation}
Therefore, it is enough to prove that
$\displaystyle{\lim_{t\rightarrow \infty}\overline\vD_\vE(E_t\setminus E)}=0$.
For a prime ideal $\p\subseteq \OO$, the $t$-th prime number $p_t$ and $M$ an integer, let us introduce the following notation.
\begin{itemize}
\item We say that $\p\succ M$ if and only if $\p$ lies over a prime greater than $M$ (Notice that, with this notation, one has that $\p\succ p_t$ implies $\p+(p_i)=\OO$ for every $i\leq t$).
\item We say that $M\succ \p$ if and only if the rational prime lying under $\p$ is less than $M$.
\end{itemize}
If $\mathcal P$ is the set of prime ideals of $\cO$, with this notation we have
\[E_t\setminus E\subseteq \bigcup_{\p\in \mathcal P: \; \p\succ p_t} \p^m\subseteq \OO^m\]
where $\p^m$ is the set $m$-tuples of elements of $\OO$ having all entries in $\p$.
It follows that
\[(E_t\setminus E)\cap \OO[B]^m\subseteq\bigcup_{\p\in \mathcal P: \; CB^n\succ\p\succ p_t} \left(\p\cap \OO[B]\right)^m\]
for $C$ a positive constant independent of $B$.
The upper bound $CB^n\succ\p$ comes from the following observation: for a fixed basis $\vE$, the norm function is a polynomial of degree $n$ in the coefficients (with respect to the basis $\vE$) of the elements of $\OO$. Therefore $N(\OO[B])\subseteq [-C B^n, C B^n]$ for a  constant $C$ depending only on the chosen basis. On the other hand, if an element of $\OO[B]$ is in $\p$ then its norm
is divisible by the rational prime $p$ lying under $\p$. This shows that there cannot exist primes $\p\succ CB^n$ containing a nonzero  element of $\OO[B]$.
We have then
\begin{align*}
\overline \vD_\vE(E_t\setminus E) & \leq \limsup_{B\rightarrow \infty}\left|\bigcup_{\p\in \mathcal P: \; CB^n\succ \p \succ p_t} \p^m\cap \OO[B]^m\right|\cdot{(2B)^{-nm}}\\
& \leq \limsup_{B\rightarrow \infty}\sum_{\p\in \mathcal P: \; CB^n\succ \p \succ p_t} |\p\cap \OO[B]|^m\cdot{(2B)^{-nm}}.
\end{align*}
By Proposition \ref{corollario}, $\displaystyle |\p\cap \OO[B]|^m=|(\p\cap \OO[B])^m|\leq \frac{(2B)^{mn}}{N(\p)^m}+c\left(\frac{2B}{c_1N(\p)^{1/n}}+1\right)^{mn-1}$. Therefore
\begin{align*}
\overline{\vD}_{\vE}(E_t\setminus E) & \leq \limsup_{B \rightarrow \infty} \sum_{CB^n \succ \p \succ p_t} |\p \cap \cO[B]|^m \cdot (2B)^{-nm}\\
&\leq \limsup_{B \to \infty} \sum_{CB^n \succ \p \succ p_t} \frac{1}{N(\p)^m} + c \left( \frac{2B}{c_1N(\p)^{1/n}}+1\right)^{mn-1} \cdot (2B)^{-nm}\\
&\leq \limsup_{B \to \infty} \sum_{CB^n >p> p_t} \frac{n}{p^m} + cn \left( \frac{2B}{c_1p^{1/n}}+1\right)^{mn-1}\cdot (2B)^{-nm}\\
&=: L_t,
\end{align*} 
where the last inequality holds because in each instance $N(\p) \geq p$ for $p$ the prime below $\p$, and above a fixed rational prime lie at most $n$ distinct primes of $\cO$.  Now our goal is to show that $L_t \to 0$ as $t\to\infty$.\\
Choose now a constant $c_1'\leq c_1$ (independent of $B$) for which $\displaystyle \frac{1}{C^{1/n}}\geq \frac{c_1'}{2}$. Notice that the sum that appears in $L_t$ is taken over primes $p$ such that $CB^n> p$, which shows that
\[\displaystyle B>\frac{1}{C^{1/n}}p^{1/n}\geq \frac{c_1'}{2}p^{1/n}.\]
It follows $\displaystyle \frac{2B}{c_1'p^{1/n}}\geq 1$ and then $\displaystyle \frac{2B}{c_1'p^{1/n}}+1\leq 2\frac{2B}{c_1'p^{1/n}}$. Therefore $L_t$ is bounded by
 \[\limsup_{B\to \infty}\sum_{p: \; CB^n>p>p_t} \frac{n}{p^m}+cn\left(\frac{4B}{c_1'p^{1/n}}\right)^{mn-1}\cdot{(2B)^{-nm}}=\]
\[ =\limsup_{B\to \infty}\sum_{p: \; CB^n>p>p_t} \frac{n}{p^m}+\frac{c'}{B\cdot p^{m-1/n}}\]
for some other constant $c'$ independent of $B$ and $p$. Now observe that
\[\limsup_{B\to\infty}\sum_{p: \; CB^n>p>p_t} \frac{n}{p^m}\leq \sum_{p>p_t}\frac{n}{p^m}\]
tends to zero when $t\to \infty$ because the series $\displaystyle \sum_p\frac{1}{p^m}$ is convergent, while for the other term one has that
\[\limsup_{B\to\infty}\sum_{CB^n>p>p_t}\frac{c'}{B\cdot p^{m-1/n}}\leq \limsup_{B\to\infty}\frac{c'}{B}\sum_{CB^n>p>p_t}\frac{1}{p}=0\]
since $\sum_{p<C B^n}\frac{1}{p}$ is asymptotic to $\log \log (CB^n)$. This concludes the proof by Equation (\ref{equazionedensita}).
\end{proof}

The following corollary produces the classical generalization of Mertens-Ces\`aro Theorem to the case of $m$-tuples of integers (presented in  \cite{Nymann}).
\begin{cor}[Extended Mertens-Ces\`aro Theorem]
The density of coprime $m$-tuples of integers is $\frac{1}{\zeta(m)}$, where $\zeta$ is the Riemann zeta function.
\end{cor}
\begin{proof}
Follows directly from Theorem \ref{MAIN} by setting $K=\vQ$.
\end{proof}

\begin{rem}
Observe that the results of Theorem \ref{MAIN} are consistent with the expectations. The obtained density is in fact independent of the basis: by symmetry, indeed, all proofs can be done by using another basis $\vB$, obtaining the same result.
In addition, Theorem \ref{MAIN} extends Mertens-Ces\`aro Theorem for algebraic integers in the following sense: over $\vZ$ one can equivalently consider the density of the set of coprime $m$-tuples of integers or coprime $m$-tuples of ideals of $\vZ$ without any relevant distinction. If one is willing to do the same in the case of algebraic integers, one has to choose in which context one wants to consider the problem: in the context of $m$-tuples of ideals, the results in \citep{BS} are satisfying while in the setting of $m$-tuples of algebraic integers, Theorem \ref{MAIN} answers the question. Curiously,
even if the set up of the problem is very different, the resulting densities match.
Future work in this direction could possibly include an analysis of the density of $r$-prime $m$-tuples of algebraic integers, extending the definition given by Sittinger in \cite{BS}.
\end{rem}
\begin{rem}
In \cite{SCHA} the author gives an asymptotic for the number of points of bounded height $B$ in the $(m-1)$-dimensional projective space.
We will briefly explain why this result goes in a similar direction of the ones in the present note.
Let $E$ be the set of coprime $m$-tuples of $\OO_K^m$. There is an action of the group of units $\OO_K^*$ on $E$ given by  
$u(c_1,...,c_m)=(uc_1,...,uc_m)$ if $u\in \OO_K^*$ and $(c_1,...,c_m)\in E$.
The natural map from $E$ to the $(m-1)$-dimensional projective space $\mathbb P_K^{m-1}$ induces an injection $\iota$ from $E/\cO_K^*$ to $\mathbb P_K^{m-1}$. When $K$ has class number one, $\iota$ is also a surjection; now one could use \cite[Theorem 3]{SCHA} to see that the number of elements of $E/\OO_K^*$ of bounded height $B$ is asymptotic to 
$C_m(K)B^m/\zeta_K(m)$ where $C_m(K)$ is a constant depending on $m$ and the number field $K$.
Schanuel gets the constants because he is essentially ``counting'' more objects. For example, when $\OO_K=\vZ[\sqrt{-5}]$ the point $Q=[1+\sqrt{-5} , 2]\in \mathbb P_K^1$ would be "counted" in the case of Schanuel result even if $Q$ is not proportional to a coprime $m$-tuple. This happens because the class number of $K$ is different from $1$.
\end{rem}
\section*{Acknowledgements}
The authors would like to thank Fabrizio Barroero for his suggestion of using Theorem \ref{latticepoints}. We also want to thank Francesco Monopoli and Reto Schnyder  for useful comments.
Moreover, we are grateful to the anonymous referee for the suggestions, which helped us to improve the structure and the content of the paper.
\bibliographystyle{plainnat}
\bibliography{biblio}{}

\end{document}